\documentclass{amsart}

\usepackage{eufrak}

\usepackage{amssymb,amsmath, amscd, mathrsfs, yfonts, bbm, graphics}

\title[] {Realization of conditionally monotone independence and monotone products of completely positive maps}

\author{Mihai Popa}\thanks{The author was partially supported in part by a Young Investigator Award, NSF
Workshop in Analysis and Probability, Texas A\&M University, 2009}

\address{${}^1$Ben Gurion University of the Negev,
 Department of Mathematics, P.O. B. 653,
 Be'er Sheva 84105, Israel}
\email{popa@math.bgu.ac.il}\address{${}^2$Institute of Mathematics, Romanian Academy,
 P.O.Box 1-764, Bucharest, RO-70700, Romania}
\DeclareMathAlphabet{\mathpzc}{OT1}{pzc}{m}{it}
\newtheorem{claim}{}[section]
\newtheorem{defn}[claim]{Definition}
\newtheorem{thm}[claim]{Theorem}
\newtheorem{lemma}[claim]{Lemma}
\newtheorem{remark}[claim]{Remark}
\newtheorem{prop}[claim]{Proposition}
\newtheorem{cor}[claim]{Corollary}

\newcommand{\cA}{\mathcal{A}}
\newcommand{\cH}{\mathcal{H}}
\newcommand{\iB}{\textit{B}}

\newcommand{\cB}{\mathcal{B}}

\newcommand{\gB}{\mathfrak{B}}
\newcommand{\gD}{\mathfrak{D}}

\newcommand{\gA}{\mathfrak{A}}

\newcommand{\cL}{\mathcal{L}}

\newcommand{\free}{\text{\huge{$\ast$}}}
\newcommand{\lra}{\longrightarrow}

\newcommand{\Free}{\underset{\small{i\in I}}{\free}}

\usepackage{amssymb}
\input xy
\xyoption{all}

\newcommand*{\longhookrightarrow}{\ensuremath{\lhook\joinrel\relbar\joinrel\rightarrow}}

\newcommand{\wE}{\widetilde{E}}

\begin{document}

\maketitle
\bibliographystyle{alpha}

\begin{abstract} The paper gives an operator algebras model for the conditional monotone independence, introduced by T. Hasebe. The construction is used to prove an embedding result for the N. Muraki's monotone product of C$^\ast$-algebras. Also, the formulas from the definition of conditional monotone independence are used to define the monotone product of maps which is shown to preserve complete positivity, a similar to the results from the case of free products. 

\end{abstract}

\section{Introduction}

 This material presents some results in monotone probability, a non-unital and non-symmetric type of non-commutative probability. More precisely, if $(\cA, \psi)$ is a non-commutative probability space and $I$ a totally ordered set, then a family of subalgebras $\{\cA_i\}_{i\in I}$ of $\cA$ is said to be monotone independent with respect to $\psi$ if for any $a_1, \dots, a_n$ with $a_k\in\cA_{i_k}$ such that $i_s\neq i_{s+1}$, the following properties are satisfied:
   \begin{enumerate}
\item[(1)]$\psi(a_1\cdots a_n)=\psi(a_1)\psi(a_2\cdots a_n)$ if $i_1>i_2$.
\item[(2)]$\psi(a_1\cdots a_n)=\psi(a_1\cdots a_{n-1})\psi(a_n)$ if $i_n>i_{n-1}$.
\item[(3)]$\psi(a_1\cdots a_n)=\psi(a_1\cdots a_{k-1}\psi(a_k)a_{k+1}\cdots a_n)$, if $i_{k-1}<i_k>i_{k+1}.$
  \end{enumerate}

   Many results from the free probabilities theory have non-trivial monotone independence analogues - the monotone Fock space, respectively bimodule of \cite{muraki} and \cite{mvPJM} are counterparts to the full Fock space (\cite{voiculescubook}, \cite{speicherhab}), the $H$- and $K$-transforms from \cite{muraki} and \cite{uwefranz} are analogue to the Voiculescu's $R$- and $S$-transforms etc.
  In \cite{hasebe}, T. Hasebe introduced the notion of conditionally monotone independence, in analogy to the notion of conditionally freeness from \cite{bozejko1}, \cite{bozejko2}. More precisely, if $\cA$ is a unital algebra endowed with two normalized linear functionals $\varphi$ and $\psi$, a family of subalgebras $\{\cA_i\}_{i\in I}$ of $\cA$ is said to be conditionally monotone independent if they are monotone independent with respect to $\psi$ and for any $a_1, \dots, a_n$ with $a_k\in\cA_{i_k}$ such that $i_s\neq i_{s+1}$, we have that
  \begin{enumerate}
\item[($1^\prime$)]$\varphi(a_1\cdots a_n)=\varphi(a_1)\varphi(a_2\cdots a_n)$ if $i_1>i_2$.
\item[($2^\prime$)]$\varphi(a_1\cdots a_n)=\varphi(a_1\cdots a_{n-1})\varphi(a_n)$ if $i_n>i_{n-1}$.
\item[($3^\prime$)]$\varphi(a_1\cdots a_n)=\varphi(a_1\cdots a_{k-1})[\varphi(a_k)-\psi(a_k)]
  \varphi(a_{k+1}\cdots a_n)+$\\
   ${}$\hspace{2.6cm} $\varphi(a_1\cdots a_{k-1}\psi(a_k)a_{k+1}\cdots a_n)$, if $i_{k-1}<i_k>i_{k+1}$
   \end{enumerate}

  In the theory of conditional freeness there is a Fock model, presented in \cite{bozejko1}. Also , there is an important connection between conditional freeness and complete positive maps: in \cite{boca1}, \cite{boca2} and \cite{dykema1}, it is shown how the relations from the definition of the conditional freeness appear in the construction of the free product of completely positive maps, which turns to  also be complete positive. The present material addresses this topics for the case of conditionally monotone independence.

  In Section 2 we will present a operator algebraic model for the conditionally monotone independence using the ``monotone Fock space'' introduced in \cite{muraki} and the ideas from the Fock model for conditionally freeness from \cite{bozejko1}. In Section 3 we construct the monotone product of maps, and using the results and some techniques from \cite{boca1} and \cite{hasebe} we prove that a monotone product of completely positive maps is completely positive. In Section 4 we define the montone product of C$^\ast$-algebras with conditional expectation, refining the construction from \cite{muraki} and prove some embedding results similar to the one from Sections 1 and 2 of \cite{dykema1}.

\section{Realization of conditionally monotone independence}\label{monotcalg}

 Let $\{(\mathcal{A}_i,\varphi_i,\psi_i)\}_{i\in I}$ be a family of $\ast$-algebras, each endowed with two states (throuout the paper  $I$ will always be a totally ordered set). If $I$ has a minimal element, $0_I$, since in the definition of the conditional monotone independence the functional $\psi_{0_i}$ does not appear, we will also suppose that $\varphi_{0_I}=\psi_{0_I}$.

 As in \cite{bozejko1}, for each $j\in I$ we  consider $\pi_j,\sigma_j:\mathcal{A}_j\lra\textit{B}(\mathcal{H}_j)$ be the $\ast$-representations given by the GNS-constructions with states $\varphi_j, \psi_j$,  respectively, i. e.
 \[
  \varphi_j(a_j)=\langle\pi_j (a_j) \xi_j, \xi_j\rangle\ \text{and}\ \psi_j(a_j)=\langle\sigma_j(a_j)\xi_j, \xi_j\rangle \hspace{2cm}
 \]
 with $a_j\in\mathcal{A}_j$ and $||\xi_j||=1$. As remarked in \cite{bozejko1}, we can always choose the same vector $\xi_j$ for both states, but by doing so we may lose the cyclicity of $\xi_j$.

  Let $(\mathcal{H},\xi)$ be the monotone product of the family $\{(\mathcal{H}_j, \xi_j)\}_{j\geq 0}$ (see \cite{muraki}, \cite{mvPJM}):
  \[
  \mathcal{H}=\mathbb{C}\xi\oplus\bigoplus_{n=1}^\infty
  \LARGE(
  \bigoplus_{i_1>\dots>i_n} \mathcal{H}_{i_1}^\circ\otimes\cdots \otimes \mathcal{H}_{i_n}^\circ
  \LARGE)
  \]
where $\mathcal{H}_j^\circ=\mathcal{H}_j\ominus \mathbb{C}\xi_j$.

  We also define
  \[
  \mathcal{H}(k)=\mathbb{C}\xi\oplus\bigoplus_{n=1}^\infty
  \big(
  \bigoplus_{\substack{i_1>\dots>i_n\\ i_1\leq k}} \mathcal{H}_{i_1}^\circ\otimes\cdots \otimes \mathcal{H}_{i_n}^\circ
  \big)
  \]
 and consider the adjointable partial isometries $V_k:\mathcal{H}\lra \mathcal{H}_k\otimes \mathcal{H}(k-1)$ given by $V_k\xi=\xi_k\otimes\xi$ and, for $f_1\otimes\cdots\otimes f_n\in\mathcal{H}_{i_1}^\circ\otimes\cdots \otimes \mathcal{H}_{i_n}^\circ$,
 \[
  V_kf_1\otimes\cdots\otimes f_n=
  \left\{
  \begin{array}{lcc}
  0,& \text{if}\ i_1>k\\
  f_1\otimes\cdots\otimes f_n, & \text{if}\ i_1=k\\
  \xi_k\otimes f_1\otimes\cdots\otimes f_n, &  \text{if}\ i_1<k.
  \end{array}
  \right.
\]

 If $T\in\textit{B}(\mathcal{H}_k)$, we define
  $\omega_k(T)=V_k^\ast ( T\otimes \text{Id}_{\cH(k)})V_k$;
  a trivial computation gives that $\omega_k(T_1T_2)=\omega_k(T_1)\omega_k(T_2)$.
   We will consider the $\ast$-representation
  $j_k:\cA_k\lra \iB(\cH)$
 \[
 j_k(a)=\omega_k(\pi_k(a))P_k\oplus\omega_k(\sigma_k(a))P_k^{\bot}
 \]
where $P_k$ is the orthogonal projection on $\mathbb{C}\xi\oplus\cH_k$.

 Finally, let  $\Phi$ be the state on $\iB(\cH)$ given by $\Phi(T)=\langle T\xi, \xi\rangle$.

 \begin{remark}\label{remark1}
  From the definition of $\omega_i$ and $\pi_i$, we have that $\Phi\circ j_i=\varphi_{i}$ for all $i\in I$.
 \end{remark}
 \begin{lemma}\label{lemma1}
  Suppose that $i_k\neq i_{k+1}$ for $1\leq k < n$, that $a_k\in\cA_{i_k}$ and \\ $A_k=j_{i_k}(a_k)$. Then
  \[
  A_1\cdots A_n\xi=\Phi(A_1\cdots A_n)\xi+\eta\ \ \ \text{for some}\ \eta\in \cH(k)^\circ=\cH(k)\ominus\mathbb{C}\xi.
  \]
 \end{lemma}
 \begin{proof} Induction on $n$. For $n=1$, we have
 \begin{eqnarray*}
 A_1\xi&=&\omega_{i_1}(\pi_{i_1}(a_1))P_{i_1}\xi\oplus\omega_{i_1}(\sigma_{i_1}(a_1))P_{i_1}^{\bot}\xi\\
 &=&\omega_{i_1}(\pi_{i_1}(a_1))\xi, \ \text{since}P_{i_1}^{\bot}\xi=0 \\
 &=& V_{i_1}^\ast [\pi_{i_1}(a_1)\xi_{i_1} ]\otimes\xi\\
 &=&V_{i_1}^\ast [\langle\pi_{i_1}(a_1)\xi_{i_1}, \xi_{i_1}\rangle\xi_{i_1} + P_{\xi_{i_1}}^\bot\pi_{i_1}(a_1)\xi_{i_1}  ]\otimes\xi\\
 &=&V_{i_1}^\ast [\varphi_{i_1}\xi_{i_1} + P_{\xi_{i_1}}^\bot\pi_{i_1}(a_1)\xi_{i_1}  ]\otimes\xi\\
 &=& \varphi_{i_1}\xi + P_{\xi_{i_1}}^\bot\pi_{i_1}(a_1)\xi_{i_1}
 \end{eqnarray*}

\noindent where $P_{\xi_{i_1}}$ is the orthogonal projection on $\mathbb{C}\xi_{i_1}$. The conclusion follows now from Remark \ref{remark1}.

 For the induction step, we first write
 $
 A_2\cdots A_n\xi=\eta_1+\eta_2+\alpha\xi
 $
 with $\eta_1\in\cH_{i_1}^\circ$, $\eta_2\in\cH(i_2)^\circ\ominus\cH_{i_1}^\circ$ and $\alpha\in\mathbb{C}$.

  The argument above gives
  $
  A_1(\alpha\xi)=\alpha(\zeta_1+\alpha_0\xi)
  $ with $\zeta_1\in\cH_{i_1}^\circ$ and $\alpha_0\in\mathbb{C}$.

  On the other hand,
  \[A_1\eta_1=P_{\xi_{i_1}}^\bot A_1\eta_1+\alpha_1\xi\ \ \text{with}\ \ \alpha_1\in\mathbb{C}\ \text{and}\ P_{\xi_{i_1}}^\bot A_1\eta_1\in\cH_{i_1}^\circ
  \]
and
\begin{eqnarray*}
A_1\eta_2&=& V_{i_1}^\ast(\sigma_{i_1}(a_1)\otimes\text{Id})V_{i_1}\eta_2\\
&=&V_{i_1}^\ast(\sigma_{i_1}(a_1)\xi_{i_1}\otimes\eta_2)\in\cH(i_1)^\circ.
\end{eqnarray*}
Summing, we obtain $A_1\cdots A_n\xi=\eta+\beta\xi$, with $\beta\in\mathbb{C}$ and $\eta\in\cH(i_1)^\circ$, and since $\Phi(T)=\langle T\xi, \xi\rangle$, the result is proved.
\end{proof}
 \begin{thm}
   With the notations from above, if $i_k\neq i_{k+1}$, $(k=1, \dots, n-1)$, and $a_k\in\cA_{i_k}$, then:
   \begin{enumerate}
   \item[(i)] for $i_1>i_2$, we have that
   \[
   \Phi\Big(j_{i_1}(a_1)\cdots j_{i_n}(a_n)\Big)=\varphi_{i_1}(a_1)\Phi\Big(j_{i_2}(a_2)\cdots j_{i_n}(a_n)\Big)
   \]
   \item[(ii)] for $i_n>i_{n-1}$, we have that
   \[
    \Phi\Big(j_{i_1}(a_1)\cdots j_{i_n}(a_n)\Big)=
    \Phi\Big(j_{k_1}(a_1)\cdots j_{k_{n-1}}(a_{n-1})\Big)\varphi_{i_n}(a_n)
   \]

   \item[(iii)] for  $i_{l-1}<i_l>i_{l+1}$ (for some $1<l<n$), we have that
   \begin{eqnarray*}
   \Phi\Big(j_{k_1}(a_1)\cdots j_{k_n}(a_n)\Big)&=&
   \Phi\Big(j_{k_1}(a_1)\cdots j_{k_{l-1}}(a_{l-1})
   \psi(a_l)
   j_{l+1}(a_{l+1})\cdots j_{k_n}(a_n))\Big)+ \\
   &&\hspace{-3cm}\Phi\Big(j_{k_1}(a_1)\cdots j_{k_{l-1}}(a_{l-1})\Big)
   \Big[\varphi_{k_l}(a_l)-\psi_{k_l}(a_l)\Big]
   \Phi\Big(j_{l+1}(a_{l+1})\cdots j_{k_n}(a_n))\Big).
   \end{eqnarray*}
\end{enumerate}
 \end{thm}
 \begin{proof}
  First, to simply the notations, we will write $A_l$ for $j_{i_l}(a_l),\ 1\leq l\leq n$.

  From Lemma \ref{lemma1} and Remark \ref{remark1}, we have that
  $A_n\xi=\eta+\varphi_{i_n}(a_n)\xi$ with $\eta\in\cH_{i_n}^\circ$. Since $i_n>i_{n-1}$, the definition of $V_{i_{n-1}}$ gives
  \begin{eqnarray*}
  \Phi(A_1\cdots A_n)&=&\langle A_1\cdots A_{n-1}\varphi_{i_n}(a_n)\xi, \xi\rangle\\
  &=&\langle A_1\cdots A_{n-1}\xi, \xi\rangle\varphi_{i_n}(a_n)
  \end{eqnarray*}
  so part (ii) is done.

  For part (i), Lemma \ref{lemma1} gives
  \[
  A_2\cdots A_n\xi=\eta+\alpha\xi,
   \]
   with $\alpha=\Phi(A_2\cdots A_n)\in\mathbb{C}$ and $\eta\in\cH(i_2)^\circ$. Since
   \begin{eqnarray*}
   A_1\eta&=&V_{i_1}(\sigma_{i_1}(a_1)\otimes\text{Id})\xi_{i_1}\otimes\eta\\
   &=&V_{i_1}(\sigma_{i_1}(a_1)\xi_{i_1}\otimes\eta)\in\cH(i_1)^\circ,
   \end{eqnarray*}
   we have that
   \begin{eqnarray*}
   \Phi(A_1\cdots A_n)&=&\langle A_1 \alpha\xi, \xi\rangle\\
   &=&\langle A_1 \xi, \xi\rangle\alpha\\
   &=&\varphi_{i_1}(A_1)\Phi(A_2\cdots A_n).
   \end{eqnarray*}

   For part (iii), write $A_{l+1}\cdots A_n\xi=\eta+\alpha\xi$, with $\eta\in\cH(l+1)^\circ$ and $\alpha=\Phi(A_{l+1}\cdots A_n)\in\mathbb{C}$. Also write
   \begin{eqnarray*}
   \pi_{i_l}(a_l)\xi_{i_l}&=&\zeta_1+\beta_1\xi_{i_l}\\
   \sigma_{i_l}(a_l)\xi_{i_l}&=&\zeta_2+\beta_2\xi_{i_l}
   \end{eqnarray*}
   with $\beta_1=\varphi_{i_l}(a_l)$, $\beta_2=\psi_{i_l}(a_l)$ and  $\zeta_1, \zeta_2\in\cH_l^\circ$.

   We have that
   \begin{eqnarray*}
   A_lA_{l+1}\cdots A_n\xi&=&A_l(\eta+\alpha\xi)\\
   &=&
   V_{i_l}(\sigma_{i_l}(a_l)\xi_{i_l}\otimes\eta)+
   V_{i_l}(\pi_{i_l}(a_l)\alpha\xi_{i_l}\otimes\xi)\\
   &=&
   V_{i_l}\big([\zeta_2+\beta_2\xi_{i_l}]\otimes\eta+
   [\zeta_1\alpha+\beta_1\alpha\xi_{i_l}]\otimes\xi\big)\\
   &=&
    \zeta_2\otimes\eta+ \beta_2\eta+\zeta_1\alpha+\beta_1\alpha\xi\\
   &=&
    \zeta_2\otimes\eta+  \zeta_2\alpha+\beta_2\eta+\beta_2\alpha\xi +(\zeta_1-\zeta_2)\alpha+(\beta_1-\beta_2)\alpha\xi
   \end{eqnarray*}
   Since $i_l>i_{l-1}$ and $\zeta_1, \zeta_2\in\cH_{i_l}^\circ$, it follows that $A_{l-1}(\zeta_2\otimes\eta+\zeta_2\alpha+(\zeta_1-\zeta_2)\alpha)=0$, therefore
   \begin{eqnarray*}
   \Phi(A_1\cdots A_n)
   &=&
   \langle A_1\cdots A_{l-1} A_l(A_{l+1}\cdots A_n\xi),\xi\rangle\\
   &=&
   \langle  A_1\cdots A_{l-1}(\beta_2[\eta+\alpha\xi]+(\beta_1-\beta_2)\alpha\xi) , \xi\rangle\\
   &=&
   \langle  A_1\cdots A_{l-1}\beta_2 A_{l+1}\cdots A_n\xi, \xi\rangle+\\
   &&\langle  A_1\cdots A_{l-1}\xi , \xi\rangle(\beta_1-\beta_2)\alpha
   \end{eqnarray*}
  hence q.e.d..
 \end{proof}

     \section{Monotone products of completely positive maps}

    Let $\{\gA_i\}_{i\in I}$ be a family of $\ast$-algebras containing a C$^\ast$-algebra $\gB$ as a common $\ast$-subalgebra and suppose that each $\gA_i$ is endowed with a projection $\psi_i:\gA_i\lra \gB$. Let now $\gD$ be another $\ast$-algebra containing $\gB$ as a $\ast$-subalgebra and suppose $\theta_i:\gA_i\lra\gD$ are a family of $\gB-\gB$ bimodule maps such that ${\theta_i}_{|\gB}=\text{Id}_{\gB}$.

  We will write $\gA_i^\circ$ for the set $\text{Ker} \psi_i\subset \gA_i$ and denote be  $\underset{\small{i\in I}}{\free}\gA_i$ the free product of $\ast$-algebras $\{\gA_i\}_{i\geq0}$ \emph{with amalgamation} over $\gB$.

    We fist need to briefly review a result from \cite{boca1}.

  \begin{defn}
   The \emph{free product} of the maps
   $\{\theta_i\}_{i\in I}$ is the map 
   \[
   \theta_\free=\underset{i\in I}{\free}\theta_i:\Free \gA_i\lra \gD
   \] 
   given by
   \begin{eqnarray*}
   \theta_\free (a_1\cdots a_n)&=&\theta_{i_1}(a_1)\cdots \theta_{i_n}(a_n)\ \text{whenever}\ a_k\in \gA_{i_k}^\circ, i_j\neq i_{j+1}\\
   {\theta_\free}_{|\gB}&=&\text{Id}_{\gB}.
   \end{eqnarray*}
  \end{defn}

 \begin{thm}\emph{(B. Foca, \cite{boca1}, Theorem 3.2)}\label{boca11}
  If, with the notations above, $\{\gA_i\}_{i\in I}$,  $\gB$, and $\gD$ are unital C$^\ast$-algebras, $\psi_i$ are projections of norm 1 and $\theta_i$ are completely positive unital maps, then $\theta_{\free}$ extends to a unital completely positive map from the universal free product of the C$^\ast$-algebras $\gA_i$ to $\gD$.
 \end{thm}

  As in \cite{hasebe}, for each $i\in I$ consider now $\widetilde{\gA_i}=\gB1\oplus\gA_i$ (direct sum of $\gB$-bimodules). If $\gB$ is unital, then $\widetilde{\gA_i}$ is a unitalization of $\gA_i$, but $1_{\widetilde{\gA_i}}\neq 1_{\gA_i}$. let $\widetilde{\gA}=\Free \widetilde{\gA_i}$ be the free product of $\ast$-algebras \emph{with} amalgamation over $\gB$. Nothe that we have the natural decomposition $\widetilde{\gA}=\gB1\oplus \gA$, where $\gA=\overline{\free}\gA_i$, the free product of $\ast$-algebras \emph{without} amalgamation over $\gB$. The algebra $\gA$ is still a $\gB$-ring, and, as a vector space, we have the identification
  \[
  \gA\cong\bigoplus_{n=1}^\infty\bigoplus_{i_1\neq i_2\cdots\neq i_n} \gA_{i_1}\otimes_{\gB}\gA_{i_2}\otimes_{\gB}\cdots \gA_{i_n}.
   \]

   \begin{defn}\label{defn1}
   The monotone product of the maps $\{\theta_i\}_{i\geq0}$ is the map\\ $\theta=\underset{i\in I}{\rhd}\theta_i:\gA\lra\gD$ given by:
 \begin{eqnarray}
 \theta(a_1\cdots a_n)&=&\theta_{i_1}(a_1)\theta(a_2\cdots a_n)\ \text{if}\ i_1>i_2\label{eq1}\\
 \theta(a_1\cdots a_n)&=&\theta(a_1\cdots a_{n-1})\theta_{i_n}(a_n)\ \text{if}\ i_n>i_{n-1}\label{eq2}\\
 \theta(a_1\cdots a_n)&=&\theta(a_1\cdots a_{k-1}\psi_{i_k}(a_k) a_{k+1}\dots a_n)+\label{eq3}\\
 &&\hspace{1cm}\theta(a_1\cdots a_{k-1})[\theta_{i_k}(a_k)-\psi_{i_k}(a_k)]\theta(a_{k+1}\cdots a_n).\nonumber
 \end{eqnarray}

   \end{defn}

\begin{prop}
 The monotone product of maps, defined above, is associative.
\end{prop}
\begin{proof}
 The proof is just a trivial (though tedious) re-writing of the argument from the scalar case in \cite{hasebe}, Theorem 3.6.
\end{proof}
\begin{prop}\label{cfree-monot}
 Let $\gA_1, \gA_2, \gD$ be $\ast$-algebras  containing a common $\ast$-subalgebra $\gB$. Suppose that $\psi_2:\gA_2\lra \gB$ is a conditional expectation and that $\theta_k:\gA_k\lra\gD, k=1,2$ are $\gB$-bimodule maps.

 Consider $\widetilde{\gA_k}=\gB1\oplus\gA_k, k=1,2$ and $\widetilde{\psi_k}, \widetilde{\theta_k}:\widetilde{\gA_k}\lra\gB$ given by ($b\in\gB$, and $a_k\in\gA_k$):
 \begin{eqnarray*}
 \widetilde{\psi_1}(b1+a_1)&=&b\\
 \widetilde{\psi_2}(b1+a_2)&=&b+\psi_2(a_2)\\
 \widetilde{\theta_k}(b1+a_k)&=&b+\theta_k(a_k).
 \end{eqnarray*}
 Then we have that
 \[
{\widetilde{\theta_2}\free\widetilde{\theta_1}}_{|\gA_1\overline{\free}\gA_2}=\theta_2\rhd\theta_1.
 \]
 \end{prop}

 \begin{proof}
  For simplicity, denote  $\theta=\theta_2\rhd\theta_1$. We just need to show that
  \[\theta(a_1\cdots a_n)=\theta_{i_1}(a_1)\cdots\theta_{i_n}(a_n)\]
  whever $a_j\in\text{Ker}(\psi_{i_j})\cap\gA_{i_j}, i_j\in\{1,2\}$ with $i_k\neq i_{k+1}$.

  We will prove the assertion by induction on $n$. The case $n=1$ is trivial.
 For the induction step, note that if $a_1$ or $a_n$ are from $\gA_2$, then the conclusion follows from Definition \ref{defn1}, relations (\ref{eq1}), (\ref{eq2}).

 If $a_1, a_n\in\gA_1$, then there exists $k\in\{2, \dots, n-1\}$ such that $a_k\in\gA_2$. Then (\ref{eq3}) implies
 \begin{eqnarray*}
 \theta(a_1\cdots a_n)&=&\theta(a_1\cdots a_{k-1}\psi_{2}(a_k) a_{k+1}\dots a_n)+\\
 &&\hspace{1cm}\theta(a_1\cdots a_{k-1})[\theta_{2}(a_2)-\psi_{2}(a_k)]\theta(a_{k+1}\cdots a_n)\\
 &=&\theta(a_1\cdots a_{k-1})\theta_{2}(a_2)\theta(a_{k+1}\cdots a_n),
 \end{eqnarray*}
 since $\psi_2(a_k)=0$, and the conclusion follows from the induction hypothesis.
 \end{proof}

 \begin{thm}\label{thmboca}
  Suppose now that $\{\gA_i\}_{i\in I}, \gD$ are unital C$^\ast$-algebras, $\gB$ is a common C$^\ast$-subalgebra of theirs containing the unit. Suppose that $\psi_i:\gA_i\lra\gB$ are positive conditional expectations and $\theta_i:\gA_i\lra\gD$ are unital, completely positive $\gB$-bimodule maps. Then the map $\underset{i\geq0}{\rhd}\theta_i$ is also a  completely positive $\gB$-bimodule map.
 \end{thm}
 \begin{proof}
 The proof relies heavily on Theorem \ref{boca11} and the associativity of the monotone product of maps.
 To simplify the writting, denote  $\theta=\underset{i\in I}{\rhd}\theta_i$ the monotone product map and $\gA=\overline{\Free}\gA_i$ the free product $\ast$-algebra.

  We need to show that, for any positive integer $n$, if $A=[a_{i,j}]_{i,j=1}^n$ is a positive element from $M_n(\gA)$ then the matrix $\theta(A)=[\theta(a_{i,j})]_{i,j=1}^n$
is also positive in $M_n(\gD)$. Each entry $a_{i,j}$ of $A$ is a finite sum
 \[ a_{i,j}=\sum_{l=1}^{N(i,j)}\alpha_l(i,j)\]
 where each $\alpha_l(i,j)$ is a reduced product from $\gA$, i. e. is written as a product of the form $a_1a_2\cdots a_m$ with $a_k\in\gA_{i_k}$, $i_s\neq i_{s+1}$.

  Let $N(A)=\text{card}\{i\in I: \text{there is a word in one of the entries of $A$ that contains}$ $\text{elements from $\gA_i$}\}$.

  We will prove the assertion by induction on $N(A)$. For $N(A)=1$, the conclusion is equivalent to the completely positivity of $\theta_1$.

  If $N(A)=2$, for $k=1,2$, let $\widetilde{\gA_k}=\gB1\oplus\gA_k$ and, as in Proposition \ref{cfree-monot}, consider the maps $\widetilde{\psi_k}, \widetilde{\theta_k}:\widetilde{\gA_k}\lra\gB$\ given by ($b\in\gB, a_k\in\gA_k$):
   \begin{eqnarray*}
 \widetilde{\psi_1}(b1+a_1)&=&b\\
 \widetilde{\psi_2}(b1+a_2)&=&b+\psi_2(a_2)\\
 \widetilde{\theta_k}(b1+a_k)&=&b+\theta_k(a_k).
 \end{eqnarray*}

  Remark that $\widetilde{\theta_k}$ are unital completely positive $\gB$-bimodule maps from the $\ast$-algebras $\widetilde{\gA_k}$ to $\gD$. First note that $1_{\gA_k}$ are projections in $\widetilde{\gA_k}$, respectively, and so are $e_k=1_{\widetilde{\gA_k}}-1_{\gA_k}$. Moreover, $\widetilde{\gA_k}=\gA_k\oplus \gB e_k$ (direct sum of C$^\ast$-algebras). If $a_k\in\gA_k, b\in\gB$, then $b1+a_k=be_k + (a_k+b1_{\gA_k})$ and $a_k+b1_{\gA_k}=\alpha_k\in\gA_k$. It follows that
 \begin{eqnarray*}
  \widetilde{\theta_k}(\alpha_k+be_k)
  &=&\widetilde{\theta_k}(b1+a_k)\\
   &=&b+\theta_k(a_k)\\
               &=&\theta_k(\alpha_k).
                \end{eqnarray*}

   Theorem \ref{boca11} implies now that $\widetilde{\theta_2}\free\widetilde{\theta_1}$ is a completely positive map from $\widetilde{\gA_2}\free\widetilde{\gA_1}$ to $\gD$, particularly from $\gA_2\overline{\free}\gA_1$ to $\gD$, and the assertion follows now from Proposition
   \ref{cfree-monot}.
   
    The induction step follows from the above argument and the associativity of the monotone product of maps. To see that, we will again need an argument from \cite{boca1}.
    
     A $\ast$-algebra $A$ is said to satisfy the \emph{Combes axiom} if for each $x\in A$ there is an $\lambda(x)>0$ such that $x^\ast x\leq \lambda(x)$. As mentioned in \cite{boca1}, \cite{boca2}, the Stinespring Dilation Theorem can be easily reformulated as follows:
     
     \emph{Let $\cA$ be a unital $\ast$-algebra satidfying the Combes axiom and let $\Phi:\cA\lra \cL(\mathcal{H})$ be a unital completely positive linear map. Then there exist a Hilbert space $K$, a $\ast$-representation $\pi:\cA\lra\cL(\mathcal{K})$ and an isometry $V\in\cL(\mathcal{H},\mathcal{K})$ such that }
     \begin{enumerate}
     \item[(i)] $\Phi(x)=V^\ast\pi(x)V$ \emph{for all}\ $x\in \cA$;
     \item[(ii)]$\mathcal{K}$ \emph{is the closed linear span of } $\pi(\cA)V\mathcal{H}$.
     \end{enumerate}   
        
  Suppose now the assertion true for $N(A)\leq n$  and let $A^\prime$ be a matrix from some $M_m(\gA)$ such that $N(A^\prime)=n+1$, that is the words summing in the entries of $A^\prime$ are containg only elements from the subalgebras $\gA_{i_1}, \gA_{i_2}, \dots, \gA_{i_{n+1}}$, with $i_1<\dots i_{n+1}$. Let $\gA(n)=\overline{\free}_{1\leq j \leq n}\gA_{i_j}$. From the induction hypothesis, we have that $\widetilde{\theta}=\underset{1\leq j \leq n}{\rhd}\theta_{i_j}$ is a completely positive map from $\gA^\prime$ to $\gD$. Take now $\widetilde{\gA(n)}$ the unitalization of $\gA^\prime$ and extend $\widetilde{\theta}$ to a completely positive map on $\widetilde{\gA^\prime}$ as above.
  
  Since $\widetilde{\gA(n)}$ is spanned by 1 and the unitaries of the C$^\ast$-algebras $\{\gA_{i_j}\}_{j=1}^n$ (in $\widetilde{\gA(n)}$ they are only partial isometries), we have that it satisfies the Combes axiom. The existence of the Stinespring dilation yields the extension of $\widetilde{\theta}$ to the greatest  C$^\ast$-algebra norm
  \[||a||=\sup\{||\pi(a)||: \pi \ast-\text{representation of}\ \widetilde{\gA(n)}\}\]
   completion of $\widetilde{\gA(n)}$. Let $\widehat{\gA(n)}$ be this C$^\ast$-algebra.    
      
   Therefore the entries of $A^\prime$ are words only in elements from $\gA_{i_{n+1}}$ and $\widehat{\gA(n)}$, which are unital C$^\ast$-algebras endowed with the compleltey positive maps $\theta_{i_{n+1}}$ and $\widetilde{\theta}$. The conclusion follows now from the argument in the case $N(A)=2$ and the associativity of the monotone products.

 \end{proof}
 
  Remark that $\overline{{\Free}}\gA_i$ satisfies the Combes axiom, since it is generated by the C$^\ast$-algebras $\{\gA_i\}_{i\in I}$. The argument from above gives then the following
 \begin{cor}
  With the notations from Theorem \ref{thmboca}, the map $\theta=\underset{i\geq0}{\rhd}\theta_i$ extends to a completely positive map on the universal free product (without amalgamation over $\gB$) C$^\ast$-algebra $\widehat{{\Free}}\gA_i$.
 \end{cor}

 
 \section{Embeddings of monotone products of C$^\ast$-algebras and completely positive maps}
 
 This section is in all regards very similar to the Sections 1 and 2 of [6]. Most of
the techniques are the same and the results are almost a verbatim translation from
the free case to the monotone case. This was to be expected, since the monotone product of Hilbert bimodules is a subspace of the free product and the partial
isometries in the definition of the monotone product of C$^\ast$-algebras are restrictions
of the unitaries from the definition of the free product. The main diference is that
we will use the construction from Section 2, while [6] is using the construction of
the conditionally monotone product from [3].
 
 \subsection{Monotone products of C$^\ast$-algebras}${}$\\

  We will use the following version of N. Muraki's construction of the monotone product of C$^\ast$-algebras.
 
 Let $\{(\gA_i, \psi_i)\}_{i\in I}$ be a family of unital C$^\ast$-algebras containing a common C$^\ast$-algebra $\gB$ with $1_{\gA_i}\in\gB$ and each $\gA_i$ endowed with a positive conditional expectation $\psi_i:\gA_i\lra\gB$ and having faithful GNS representations.

 We let $E_i=L^2(\gA_i, \psi_i), \psi_i=\widehat{1_{\gA_i}}\in E_i, E_i=\xi_i\gB\oplus E_i^\circ$. Similarly to the previous section, consider the Hilbert $\gB$-bimodules
 \[
 E=\xi\gB\oplus\bigoplus_{n=1}^\infty
  \left(
  \bigoplus_{i_1>\dots>i_n} E_{i_1}^\circ\otimes_\gB\cdots \otimes_\gB E_{i_n}^\circ
  \right)
 \]

and

\[
  E(k)=\xi\gB\oplus\bigoplus_{n=1}^\infty
  \Big(
  \bigoplus_{\substack{i_1>\dots>i_n\\ i_1\leq k}} E_{i_1}^\circ\otimes_{\gB}\cdots \otimes_{\gB}E_{i_n}^\circ
  \Big).
  \]

   Remark that we can define $\widetilde{V}_k:E\lra E_{k}\otimes_\gB E(k-1)$ similarly to the operators $V_k$ from the previous sections; they are adjointable partial isometries (considering the norm induced by the C$^\ast$-norm of $\gB$). For $a\in\cA_k$, define \[
   j_k(a)=\widetilde{V}_k^\ast(a\otimes Id)\widetilde{V}_k\in\mathcal{L}(E).
   \]
   Finally let $\gA$ be the C$^\ast$-algebra generated by $\{j_i(\gA_i)\}_{i\in I}$ in $\mathcal{L}(E)$ and let $\psi:\mathcal{L}(E)\lra \cB$ be the functional given by $\psi(T)=\langle T\xi, \xi\rangle$.

   We will call the pair $(\gA, \psi)=\underset {\scriptscriptstyle{i\in I}}{\rhd}(\gA_i, \psi_i)$ the monotone product of the family of C$^\ast$-algebras $\{(\gA_i, \psi_i)\}_{i\in I}$.

   The following property was shown in \cite{muraki} for the case $\gB=\mathbb{C}$ and in \cite{mvPJM} for the general setting:
   \begin{prop}\label{prop41}
   The functional $\psi$ from above is a conditional expectation with respect to which the subalgebras $\{j_i(\gA_i)\}_{i \in I}$ are monotone independent, i. e. for any $a_k\in j_{i_k}(\gA_{i_k}), 1\leq i\leq n$ such that $i_s\neq i_{s+1}$, we have:
   \begin{enumerate}
   \item[(a)]$\psi(a_1\cdots a_n)=\psi_{i_1}(a_1)\psi(a_2\cdots a_n)\  \text{if}\ i_1> i_2\ $
   \item[(b)]$\psi(a_1\cdots a_n)=\psi(a_1\cdots a_{n-1})\psi_{i_n}(a_n)\  \text{if}\ i_n> i_{n-1}\ $
    \item[(c)]$\psi(a_1\cdots a_n)=\psi(a_1\cdots a_{k-1}\psi_{i_k}(a_k)a_{k+1}\cdots a_n) \ \text{if}\ i_{k-1}< i_k>i_{k+1}.$
 \end{enumerate}
   \end{prop}

   \begin{remark}\label{remark42}
    Actually the subalgebras $\{j_i(\gA_i)\}_{i\in I}$ are satisfying a stronger condition than (b) and (c) from the above Proposition. If $k<l$ and $a\in j_k(\gA_k), b\in j_l(\gA_l)$, then
     \[ab_{|E\ominus E_l^\circ\otimes E(l-1)}=a\psi(b)_{|E\ominus E_l^\circ\otimes E(l-1)}\]
     Particularly,  $a_1a_2a_3=a_1\psi(a_2)a_3$ whenever $a_i\in j_{k_i}(\gA_{k_i})$ with $k_1<k_2>k_3$.
   \end{remark}
   \begin{proof}
  It suffices to show that $ab\eta=a\psi(b)\eta$ for all  $\eta=f_1\otimes_{\gB}\cdots\otimes_{gB} f_n$, with $f_j\in E_{k_j}^\circ$ such that $l\neq k_1>\cdots k_n$.

   If $k_1>l$, then both sides are zero. If $k_1<l$, then
   \[ab\eta=a\widetilde{V_l}^\ast((b\otimes \text{Id})\xi_l\otimes\eta)=a\widetilde{V_l}^\ast(\psi(b)\xi_l+P_{\xi_l}^\bot b\xi_l)\otimes\eta=a\psi(b)\eta.\]
   The last part follows from the fact that $j_i(\gA_i)(E)\subseteq E(i)$.
   \end{proof}
   
\begin{remark}
${}$
\end{remark}
  The above construction can easily be modified to obtain a representation of the free product $\ast$-algebra of the family $\{\gA_i\}_{i\in I}$ that satisfy the relations (a)-(c) from Proposition \ref{prop41} without the more restrictive condition from Remark \ref{remark42}.  With the above notations, let $(\mathcal{E}, \xi)$ be the free product bimodule of the family $\{E_i, \xi_i\}_{i\in I}$, that is
  \[ \mathcal{E}=\gB\xi\oplus\bigoplus_{n=1}^\infty
  \LARGE(
  \bigoplus_{i_1\neq\dots\neq i_n} E_{i_1}^\circ\otimes\cdots \otimes E_{i_n}^\circ
  \LARGE)
  \]
where $E_j=E_j^\circ\oplus \gB\xi_j$.

  We also define
  \[
  \mathcal{E}(k)=\gB\xi\oplus\bigoplus_{n=1}^\infty
  \big(
  \bigoplus_{\substack{i_1\neq\dots\neq i_n\\ i_1\leq k}} E_{i_1}^\circ\otimes\cdots \otimes E_{i_n}^\circ
  \big)
  \]
 and consider the partial isometries $W_k:\mathcal{E}\lra E_k\otimes \mathcal{E}(k-1)$ given by $W_k\xi=\xi_k\otimes\xi$ and, for $f_1\otimes\cdots\otimes f_n\in E_{i_1}^\circ\otimes\cdots \otimes E^\circ_{i_n}$,
 \[
  W_kf_1\otimes\cdots\otimes f_n=
  \left\{
  \begin{array}{lcc}
  0,& \text{if}\ i_1>k\\
  f_1\otimes\cdots\otimes f_n, & \text{if}\ i_1=k\\
  \xi_k\otimes f_1\otimes\cdots\otimes f_n, &  \text{if}\ i_1<k.
  \end{array}
  \right.
\]

 For $T\in\gA_i\subseteq \cL(E_i)$, define $u_i(T)=W_k^\ast (T\otimes \text{Id}_{\mathcal{E}(k-1)})W_k$ and $\psi(\cdot)=\langle \cdot\xi, \xi\rangle$.
 Since $\underset{i\in I}{\rhd}E_i=E$ is a sub-bimodule of $\mathcal{E}$ and $u_i(a)_{|E}=j_i(a)$, it follows that Proposition \ref{prop41}  holds true also for the family $\{u_i(\gA_i)\}_{i\in I}$.
 
 To see that $\{u_i(\gA_i)\}_{i\in I}$ do not satisfy the relations from Remark \ref{remark42}, consider $i_1<i_2>i_3$ from $I$ and for $j=1,2,3$ take $a_j\in\gA_{i_j}$ such that $\widehat{a_j^\ast}\in E_{i_j}^\circ$ (that is $\psi(u_{i_j}(a_j))=0$). Consider also $f_2=\widehat{a_2^\ast}$ and $f_3=\langle f_2, f_2\rangle\widehat{a_3^\ast}$. 
 
  Denoting $A_j=u_{i_j}(a_j)$, we have that $\psi(A_2)=0$, hence $A_1\psi(A_2) A_3=0$. On the other hand, since $\langle a_3f_3, \xi\rangle=\langle f_3, f_3\rangle\neq0$, we have that 
  $a_3f_3=\zeta+\langle f_3, f_3\rangle$ with $\zeta\in E_{i_3}^\circ$. Therefore
  \begin{eqnarray*}
  A_2A_3f_3\otimes f_2&=&a_2(\zeta\otimes f_2 +\langle f_3, f_3\langle f_2\\
  &=&\widehat{a_2}\otimes\zeta\otimes f_3+A_2\langle f_3, f_3\rangle f_2.
  \end{eqnarray*}
 Since $\widehat{a_2}\in E_{i_2}^\circ$ and $i_1<i_2$, we have that
 \begin{eqnarray*} 
 A_1A_2A_3f_3\otimes f_2
 &=&
 A_1A_2\langle f_3, f_3\rangle f_2\\
 &=&
 \widehat{a_1}\langle A_2\langle f_3, f_3\rangle f_2, \xi\rangle\\
 &=&\widehat{a_1}\langle \langle f_2, f_2\rangle\widehat{a_3^\ast}, \langle f_2, f_2\rangle\widehat{a_3^\ast}\rangle\neq 0.
  \end{eqnarray*}
   
    \begin{lemma}(\emph{see \cite{dykema1}, Lemma 1.1})\label{dykema1}
    Let $\{\gA_i\}_{i\in I}$ be a family of unital C$^\ast$-algebras containing a common unital C$^\ast$-subalgebra $\gB$ and having conditional expectations $\psi_i:\gA_i\lra \gB$ whose GNS representations are faithful. Let
    \[ (\gA,\psi)=\underset{i\in I}{\rhd} (A_i, \psi_i)\]
    be their monotone product of C$^\ast$-algebras as defined in Section \ref{monotcalg}. Then for every $i_0>0$ there exists a conditional expectation $\Psi_{i_0}:\gA\lra\gA_{i_0}$ such that ${\Psi_{i_0}}_{|\gA_{i}}=\psi_i$ for every
    $i\neq i_0$ and, if $a_k\in\gA_{i_k}, i_s\neq i_{s+1}$, then
    \begin{equation}\label{lem01rel}
   \Psi_{i_0}(a_1\cdots a_n)=\left\{
   \begin{array}{ll}\Psi_{i_0}(a_1)\Psi_{i_0}(a_2\cdots  a_n)& \text{if}\ i_1>i_2\\
   \Psi_{i_0}(a_1\cdots \psi_{i_k}(a_k)\cdots a_n)&  \text{if}\ i_{k-1}<i_k>i_{k+1}\\
   \Psi_{i_0}(a_1\cdots a_{n-1})\Psi_{i_0}(a_n)& \text{if}\ i_n>i_{n-1}\\
   \end{array}
   \right..
    \end{equation}
    
   \end{lemma}

   \begin{proof}
    Let $E_i=L^2(\gA_i, \psi_i), \xi_i=\widehat{1_{\gA_i}}\in E_i$, $E_i=\xi_i\gB\oplus E_i^\circ$. By construction, the algebra $\gA$ acts on the Hilbert bimodule $(E,\xi)=\underset{i\in I}{\rhd}(E_i, \xi_i)$. Identify the submodule $\xi\gB\oplus E_{i_0}^\circ$ with $E_{i_0}$ and let $Q_{i_0}:E\lra E_{i_0}$ be the projection. Then $\Psi_{i_0}(x)=Q_{i_o}x Q_{i_0}$ has the desired properties.
   \end{proof}

   \begin{remark}
   \end{remark}
    With the notations from Section \ref{monotcalg}, consider
     \[
    F=\gA_{i_0}\oplus \bigoplus_{n\geq1}\bigoplus_{i_1>\cdots i_n\neq i_0} E^\circ_{i_1}\otimes_{\gB}\cdots E^\circ_{i_n}\otimes_{\gB}\otimes_{\gB}\gA_{i_0}.
    \]
    Then $\Psi_{i_0}=\langle\cdot 1_{\gA_{i_0}}, 1_{\gA_{i_0}}\rangle$.

    Let $\rho:\gA_{i_0}\lra \cL(K)$ be a unital $\ast$-homomorphism for some Hilbert space $K$. Then $\rho$ induces a $\ast$-homomorphism $\rho|^\gA:\gA\lra\cL(F\otimes_\rho K)$ determined by its restrictions $\rho_i=\rho_{|\gA_i}\lra\cL(F\otimes_\rho K)$ given as follows.

      Writing $\mathcal{K}=F\otimes_\rho K$, we have
      \[
        \mathcal{K}=K\oplus \bigoplus_{n\geq1}\bigoplus_{i_1>\cdots i_n\neq i_0} E^\circ_{i_1}\otimes_{\gB}\cdots E^\circ_{i_n}\otimes_{\gB}\otimes_\rho K.
\]
 Consider the Hilbert spaces
 \begin{eqnarray*}
 \mathcal{K}(i)&=&(\eta_i\gB\otimes_{\rho_{|\gB}}K)\oplus\bigoplus_{n\geq1}
 \bigoplus_{\substack{i_1>\cdots i_n\\ i>i_1, i_n\neq i_0}} E^\circ_{i_1}\otimes_{\gB}\cdots E^\circ_{i_n}\otimes_{\gB}\otimes_\rho K.\ \text{if}\ i\neq i_0\\
 \mathcal{K}(i_0)&=&\bigoplus_{n\geq1}
 \bigoplus_{\substack{i_1>\cdots i_n\\ i_0>i_1, i_n\neq i_0}} E^\circ_{i_1}\otimes_{\gB}\cdots E^\circ_{i_n}\otimes_{\gB}\otimes_\rho K.\ \\
 \end{eqnarray*}
where $\eta_i\gB$ is just the Hilbert $\gB$-bimodule $\gB$ with identity element denoted by $\eta_i$. If $i\neq i_0$, consider the partial isometry
$W_i:E_i\otimes_{\gB}\mathcal{K}(i)\lra\mathcal{K}$ given by
\begin{eqnarray*}
\xi_i\otimes(\eta_i\otimes v)&\mapsto& v\\
\zeta_i\otimes(\eta_i\otimes v)&\mapsto& \zeta\otimes v\\
\xi_i\otimes(\zeta_1\otimes\cdots\otimes\zeta_n\otimes v)&\mapsto&\zeta_1\otimes\cdots\otimes\zeta_n\otimes v\\
\zeta\otimes(\zeta_1\otimes\cdots\otimes\zeta_n\otimes v)&\mapsto&\zeta\otimes\zeta_1\otimes\cdots\otimes\zeta_n\otimes v
\end{eqnarray*}
for all $v\in K$, $\zeta\in E_i^\circ, \zeta_j\in E_{i_j}^\circ$. Then for every $a\in\gA_i$ with $i\neq i_0$ we have
\[\rho_i(a)=W_i(a\otimes \text{Id}_{\mathcal{K}(i)})W_i^\ast.\]
Similarly, for $i_0$, we define the partial isometry $W_{i_0}:K\oplus(E_{i_0}\otimes_{\gB}\mathcal{K}({i_0}))\lra\mathcal{K}$ given by
\begin{eqnarray*}
v\oplus 0&\mapsto&v\\
0\oplus(\xi_{i_0}\otimes(\zeta_1\otimes\cdots\otimes \zeta_n\otimes v)
&\mapsto&
\zeta_1\otimes\cdots\otimes \zeta_n\otimes v\\
0\oplus(\zeta\otimes(\zeta_1\otimes\cdots\otimes \zeta_n\otimes v)
&\mapsto&
\zeta\otimes\zeta_1\otimes\cdots\otimes \zeta_n\otimes v.
\end{eqnarray*}
Then $\rho_{i_0}(a)=W_{i_0}(\rho(a)\oplus(a\otimes 1_{\mathcal{K}(i_0)}))W_{i_0}^\ast$. Note that the above description is related to the construction  from Section 2.

   \subsection{Embeddings of monotone products of C$^\ast$-algebras and completely positive maps}${}$\\

   \begin{prop}\emph{(see \cite{dykema1}, Theorem 1.3)}\label{dykema111} Let $\gB\subseteq\widetilde{\gB}$ be a (not necessarily unital) inclusion of unital C$^\ast$-algebras. For each $i\in I$ , suppose
   \[
   \begin{array}{ccccc}
   1_{\widetilde{\gA_i}} & \in & \widetilde{\gB}& \subseteq &\widetilde{\gA_i}\\
                         &     & \cup           &           &  \cup \\
      1_{\gA_i} & \in & \gB& \subseteq &{\gA_i}\\
   \end{array}
   \]
   are inclusions of c$^\ast$-algebras.  Suppose that $\widetilde{\psi_i}:\widetilde{\gA_i}\lra \widetilde{\gB}$ are conditional expectations such that $\widetilde{\psi_i}(\gA_i)\subseteq \gB$ and assume that $\widetilde{\psi_i}$ and the restrictions $\widetilde{\psi_i}_{|\gA_i}$ have faithful GNS representations. Let
   \begin{eqnarray*}
   (\widetilde{\gA}, \widetilde{\psi})&=&\underset{i\in I}{\rhd} (\widetilde{\gA_i}, \widetilde{\psi_i})\\
   ({\gA}, {\psi})&=&\underset{i\in I}{\rhd} ({\gA_i}, {\psi_i})
   \end{eqnarray*}
   be the monotone products of C$^\ast$-algebras. Then there is a unique $\ast$-homomorphism $\kappa:\gA\lra\widetilde{\gA}$ such that for every $i\in I$ the diagram
   \[
   \begin{array}{ccccccc}
   \widetilde{\gA_i}&\longhookrightarrow & \widetilde{\gA}\\
                \cup&                     &\hspace{.3cm}\uparrow \kappa\\
     {\gA_i}&\longhookrightarrow & {\gA}\\
   \end{array}
   \]
   commutes, where the horizontal arrows are the inclusions arising from the monotone product construction. Moreover, $\kappa$ is necessarily injective.
   \end{prop}

\begin{proof}
 Note that, since $\gA$ is generated by $\bigcup\gA_i$, it is clear that $\kappa$ will be unique if it exists. Also, we can suppose that the inclusions $\gB\subseteq\widetilde{\gB}$ and $\gA\subseteq\widetilde{\gA}$ are unital: if $1_{\gB}\neq 1_{\widetilde{\gB}}$, then we may replace $\gB$ by $\gB+\mathbb{C}(1_{\widetilde{\gB}}-1_{\gB})$ and each $\gA_i$ by $\gA_i+\mathbb{C}(1_{\widetilde{\gA_i}}-1_{\gA_i})$.

 Let
$(\widetilde{\pi}_i, \wE_i, \widetilde{\xi}_i)=\text{GNS}(\widetilde{\gA_i}, \widetilde{\psi_i}), \ \ \
  ({\pi}_i, \wE_i, {\xi}_i)=\text{GNS}({\gA_i}, {\psi_i})
 $\  and
 $(\wE, \widetilde{\xi})=\underset{i\in I}{\rhd}(\wE_i, \widetilde{\xi_i})$, respectively
 $(E, {\xi})=\underset{i\in I}{\rhd}(E_i, {\xi_i}).$

 The inclusion $\gA_i\hookrightarrow \widetilde{\gA_i}$ gives an inner-product-preserving isometry of Banach spaces $E_i\hookrightarrow \wE_i$
 sending $\xi_i$ to $\widetilde{\xi}_i$ and $E_i^\circ$ to a subspace of $\wE_i^\circ$ and allowing, for each $i_1>\cdots>i_n$,  a cannonical identification of
 \[
 E_{i_1}^\circ\otimes_{\gB}\cdots\otimes_{\gB} E_{i_{p-1}}^\circ\otimes_{\gB}
 \wE_{i_p}^\circ\otimes_{\widetilde{\gB}}\cdots
 \otimes_{\widetilde{\gB}}\wE_{i_n}^\circ
 \]
 with the a closed subspace of
 $
 \wE_{i_1}^\circ\otimes_{\widetilde{\gB}}\cdots \otimes_{\widetilde{\gB}}
 \wE_{i_n}^\circ$. We will identify $E$ with a subspace of $\wE$ as follows:
 \[
 E\cong\widetilde{\xi}\gB\oplus\bigoplus_{n=1}^\infty
  \LARGE(
  \bigoplus_{i_1>\dots>i_n} E_{i_1}^\circ\otimes_\cB\cdots \otimes_\cB E_{i_n}^\circ
  \LARGE)\subset\wE.
 \]

  Let now $\textgoth{A}=\underline{\free}\gA_i$ be the universal algebraic free product \emph{without} amalgamation. Let $\sigma:\textgoth{A}\lra\cL(E)$,  respectively $\widetilde{\sigma}:\textgoth{A}\lra\cL(\wE)$ be the homomorphism extending the homomorphisms $\pi_i:\gA_i\lra\cL(E)$, respectively $\widetilde{\pi}_{i|\gA_i}:\gA_i\lra\cL(\wE)$ (particularly, $\overline{\sigma(\textgoth{A})}=\gA$).

  In order to show that $\kappa$ exists, it suffices to show that $||\widetilde{\sigma}(x)|| \leq ||\sigma(x)||$, for all $x\in\textgoth{A}$.

   Note that $||\widetilde{\sigma}(x)|| \geq ||\sigma(x)||$ for all $x\in\textgoth{A}$, since the subspace $E$ of $\wE$ is invariant under $\widetilde{\sigma}(\textgoth{A})$ and $\widetilde{\sigma}_{|E}=\sigma$. Henceforth, if $\kappa$ exists, then it is injective.

   Let $\tau$ be a faithful representation of $\widetilde{\gB}$ on a Hilbert space $\mathcal{W}$, then consider the Hilbert space $\wE\otimes_\tau\mathcal{W}$ and let $\widetilde{\lambda}:\cL(\wE)\lra \cL(\wE\otimes_{\tau}\mathcal{W})$ be the $\ast$-homomorphism given by $\widetilde{\lambda}(x)=x\otimes 1_\mathcal{W}$. $\widetilde{\lambda}$ is faithful, hence it will suffice to show that $||\widetilde{\lambda}\circ \widetilde{\sigma}(x)||\leq ||\sigma(x)||$ for all $x\in\textgoth{A}$.

    We will show that $\widetilde{\lambda}\circ \widetilde{\sigma}$  decomposes as a direct sum of subrepresentations, each of which is of the form $(\nu|^\gA)\circ\sigma$, where $\nu|^\gA$ is the $\ast$-representation of $\gA$ induced from a representation $\nu$ of some $\gA_i$.

    For $n>0$ and $i_1>\cdots i_n$ and $1\leq p \leq n$, consider the Hilbert space\\
  $H^{(i_1,\dots, i_n)}_p=E_{i_1}^\circ\otimes_{\gB}\cdots
    \otimes_{\gB} E_{i_{p-1}}^\circ\otimes_{\gB}
    K_{i_p}\otimes_{\widetilde{\gB}}\wE_{i_{p+1}}^\circ \otimes_{\widetilde{\gB}}
    \cdots
 \otimes_{\widetilde{\gB}}\wE_{i_n}^\circ\otimes_\tau\mathcal{W}$
 defined as
   \begin{eqnarray*}
    E_{i_1}^\circ\otimes_{\gB}
 \cdots
 \otimes_{\gB} E_{i_{p-1}}^\circ\otimes_{\gB}
    \wE^\circ_{i_p}\otimes_{\widetilde{\gB}}
    &\cdots&
  \otimes_{\widetilde{\gB}}\wE_{i_n}^\circ\otimes_\tau\mathcal{W}\ominus\\
 E_{i_1}^\circ\otimes_{\gB}
 &\cdots&
 \otimes_{\gB} E_{i_{p-1}}^\circ\otimes_{\gB}
    E^\circ_{i_p}\otimes_{\widetilde{\gB}}
    \cdots
 \otimes_{\widetilde{\gB}}\wE_{i_n}^\circ\otimes_\tau\mathcal{W}.
    \end{eqnarray*}

    Then
    \[
    \wE\otimes_\tau\mathcal{W}=(E\otimes_{\tau_{|\gB}}\mathcal{W})\oplus
    \LARGE(
    \bigoplus_{n=1}^\infty\bigoplus_{\substack{i_1>\cdots>i_n\\ 1\leq p \leq n}}
    H^{(i_1,\dots, i_n)}_p \LARGE).
    \]

     As previously mentioned, $\widetilde{\sigma}(\textgoth{A})E\subseteq E$ and
     $\widetilde{\sigma}_{|E}=\sigma$, so $E\otimes_{\tau_{|\gB}}\mathcal{W}$ is invariant under $\widetilde{\lambda}\circ\widetilde{\sigma}(\textgoth{A})$, and
     $||\widetilde{\lambda}\circ\widetilde{\sigma}(x)_{|E\otimes_\tau\mathcal{W}}||
     =||\sigma(x)||$  for all $x\in\textgoth{A}$.

     Define $\widetilde{\mathcal{W}}(i_1, \cdots, i_n)=
     \overline{
     \widetilde{\lambda}\circ\widetilde{\sigma}(\textgoth{A})H_1^{(i_1,\dots, i_n)}
     }$. Since $\widetilde{\pi}_i(\gA_i)E_i\subseteq E_i$ we have that
     \[
     \widetilde{\mathcal{W}}(i_1, \cdots, i_n)= H_1^{(i_1,\dots, i_n)}\oplus
     \LARGE( \bigoplus_{l\geq1}
     \bigoplus_{\substack{k_1>\cdots k_l\\k_l>i_1}}E_{k_1}^\circ\otimes_{\gB}\cdots E_{k_s}
     \otimes_{\gB}H_1^{(i_1,\dots, i_n)}
     \LARGE).
     \]

      Thus,
      \[
      \wE\otimes_\tau\mathcal{W}=(E\otimes_{\tau_{|\gB}}\mathcal{W})\oplus
      \bigoplus_{n\geq1}\bigoplus_{i_1>\cdots i_n} \widetilde{\mathcal{W}}(i_1,\dots, i_n);
      \]
  Hence to prove the theorem it will suffice to show that for all $i_1>\cdots > i_n$
  and all $x\in \textgoth{A}$,
  \begin{equation}\label{dykema6}
  ||\widetilde{\lambda}\circ\widetilde{\sigma}(x)_{|\widetilde{\mathcal{W}}(i_1,\dots, i_n)}||\leq ||\sigma(x)||.
  \end{equation}

  But letting $\nu:\gA_{i_1}\lra\cL(H_1^{(i_1,\dots, i_n)})$ be the $\ast$-homomorphism
  \[
  \nu(a)=(\widetilde{\pi}_{i_1}(a)\otimes 1_{
  \wE_{i_2}^\circ \otimes_{\widetilde{\gB}}
    \cdots
 \otimes_{\widetilde{\gB}}\wE_{i_n}^\circ\otimes_\tau\mathcal{W}
  }\Large)_{|H_1^{(i_1,\dots, i_n)}}
  \]
  and considering $\nu|^\gA$ be the representation of $\gA$ induced from $\nu$ with respect to the conditional expectation $\Psi_{i_1}:\gA\lra\gA_{i_1}$ found in Lemma \ref{dykema1},
  it is straightforward to check that
  \[
  \widetilde{\lambda}\circ\widetilde{\sigma}_{|\widetilde{\mathcal{W}}(i_1, \cdots, i_n)}=\LARGE(\nu|^\gA\LARGE)\circ\sigma,
  \]
  which, in turn implies (\ref{dykema6}).
\end{proof}

   \begin{thm}\emph{(see \cite{dykema1}, Theorem 2.2)}\label{thm11}
   Let $\cB$ be a unital C$^\ast$-algebra, and for every $i\geq 0$ let $\cA_i$ and $D_i$ be unital C$^\ast$-algebras containing copies of $\cB$ as unital c$^\ast$-subalgebras and having conditional expectations $\phi_i:\cA_i\lra\cB$, respectively $\psi_i:D_i\lra\cB$, whose GNS representations are faithful. Suppose that for each $i\geq0$ there is a unital completely positive map $\theta_i:\cA_i\lra D_i$ that is also a $\cB$ bimodule map and satisfies $\psi_i\circ\theta_i=\phi_i$. Denote
   \begin{eqnarray*}
   (\cA,\phi)&=&\underset{\scriptscriptstyle{i\in I}}{\rhd}(\cA_i, \phi_i)\\
   (D,\psi)&=&\underset{\scriptscriptstyle{i\in I}}{\rhd}(D_i, \psi_i)
   \end{eqnarray*}
    the monotone products of C$^\ast$-algebras. Then there is a unital completely positive map $\theta:\cA\lra D$ such that for all $i\geq 0$ the diagram

      \[
 \begin{xy}
(15,30)*+{\frak{A}_i}="v1";(45,30)*+{\frak{D}_i}="v2";%
(30,15)*+{\frak{B}}="v3";
(15,0)*+{\frak{A}}="v6";(45,0)*+{\frak{D}}="v7";%
{\ar@{->} "v1"; "v6"};%
{\ar@{->}^{\theta_i} "v1"; "v2"};
{\ar@{->} "v1"; "v3"};%
{\ar@{->} "v2"; "v7"};
{\ar@{->}|{\psi} "v7"; "v3"};%
{\ar@{->}|{\phi_i} "v1"; "v3"};%
{\ar@{->}|{\psi_i} "v2"; "v3"};%
{\ar@{->}|{\phi} "v6"; "v3"};%
{\ar@{-->}_{\theta} "v6"; "v7"}
\end{xy}
 \]
 commutes, where the vertical arrows are the (non-unital) inclusions arising from the monotone product construction. Moreover, the mapping $\theta$ satisfies:
 \begin{enumerate}
   \item[(i)] $
   \theta(a_1\cdots a_n) =\theta(a_1)\theta (a_2\cdots a_n)
   $,\ if $i_1>i_2$;
   \item[(ii)] $
    \theta(a_1 \cdots a_n)=
    \theta(a_1 \cdots a_{n-1})\theta(a_n)
   $,\ if $i_n>i_{n-1}$;

   \item[(iii)]
   $
   \theta(a_1\cdots a_n)=
   \theta(a_1\cdots a_{l-1}\cdot
   \phi_{i_l}(a_l)\cdot
   a_{l+1}\cdots a_n)
   $
   if  $i_{l-1}<i_l>i_{l+1}$.
\end{enumerate}
   \end{thm}

   \begin{proof}
 Let $ (\pi_I, E_i, \xi_i)=\text{GNS}(D_i,\psi_i)$, and $(E, \xi)=\rhd (E_i,\xi_i)$, as in the previous section.

 Consider the Hilbert $\cB$-bimodule $F_i=\cA_i\otimes_{\pi_i\circ \theta_i} E_i$ with the distinguished element $\eta_i=1\otimes\xi_i\in F_i$. The mapping $\theta_i$ restricts to the identity map on $\cB$, so in $F_i$ we have that $b\otimes \zeta=1\otimes (b\zeta)$ for every $b\in\cB$. Consider the unital $\ast$-homomorphism $\sigma_i:\cA_i\lra \cL(F_i)$
 \[
 \sigma_i(a_1)(a_2\otimes\zeta)=(a_1a)\otimes \zeta,\ \text{for all}\ a_1,a\in\cA_1, \zeta\in E_i
 \]
 and the map $\rho_i:\cL(F_i)\lra \cB$ given by $
 \rho_i(x)=\langle \eta_i,x\eta_i\rangle$.
 As in \cite{dykema1} we have that, identifying $\cB$ with $\sigma_i(\cB)\subseteq\cL(F_i)$, the map $\rho_i:\cL(F_i)\lra\cB$ is a conditional expectation, that $L^2(\cL(F_i),\rho_i)\cong F_i$, that the GNS representation of $\rho_i$ is faithful on $\cL(F_i)$ and that $\rho_i\circ\sigma_i=\phi_i$.

 Take now $(\mathcal{M},\rho)=\rhd(\cL(F_i),\rho_i)$ and note that (see \cite{mvPJM}) $\mathcal{M}\subseteq \cL(F)$, where $(F,\eta)=\rhd(F_i,\eta_i)$. By Proposition \ref{dykema111} there is a $\ast$-homomorphism $\sigma:\cA\lra\mathcal{M}$ such that $\sigma_{|\cA_i}=\sigma_i$.

  Consider the operator $v_i:E_i\lra F_i$ given by $\zeta\lra 1\otimes \zeta$. As shown in \ref{dykema1}, proof of Theorem 2.2, we have that $v_i$ is an adjointable (its adjoint being the operator $F_i\lra E_i$ sending $a\otimes\zeta$ to $\theta_i(a)\zeta$), that  $v_i(E_i^\circ)\subseteq F_i^\circ$ and $v_i^\ast v_i=1$. Since $\theta_i$ is a left $\gB$-bimodule map, $v_i(b\zeta)=1\otimes (b\zeta)=b\otimes \zeta=b(v(\zeta))$, for all $b\in\gB$, $\zeta\in E_i$.

  Taking direct sum of operators $v_{i_1}\otimes\cdots \otimes v_{i_n}$, we get that $v\in\cL(E)$ such that $\langle v\zeta, \zeta\rangle =\langle\zeta, \zeta\rangle$ for every $\zeta\in E$, that $v\xi=\eta$ and
  \[v(\zeta_1\otimes\cdots \otimes \zeta_n)=(v_{i_1}\zeta_1)\otimes\cdots \otimes (v_{i_n}\zeta_n), \ \text{whenever}\ \zeta_j\in E_{i_j}, i_1>\cdots >i_n.\]

  Let $\theta:\gA\lra \cL(E)$ be the unital completely positive map
  \[\theta(x)=v^\ast\sigma(x) v.\]

  We will show that $\theta$ satisfies the Theorem. In order to show that the diagram commutes, let $w_i:E\lra E_i\otimes_{\gB} E(i-1)$ and $y_i:F\lra F_i\otimes_\gB F(i-1)$ be the partial isometries that we used in the monotone product construction for the inclusions $\gA_i\hookrightarrow \gA$, respectively $\cL(F_i)\hookrightarrow\mathcal{M}$. Exactly as in \cite{dykema1}, note that $v_i(E(i-1))\subseteq F(i-1)$ and that $y_iv=(v_i\otimes v_{|E_i})w_i$.

  Hence, for $a\in\gA_i$, we have that
  \begin{eqnarray*}
  \theta(a)&=&v^\ast\sigma(a)v=v^\ast\sigma_i(a)v=v^\ast y_i[\sigma_i(a)\otimes 1_{F(i-1)}]y_iv\\
  &=&w_i^\ast[v_i\sigma_i(a)v_i\otimes (v_{|E(i-1)})^\ast v_{|E(i)}]w_i\\
  &=&w_i^\ast[\theta_i(a)\otimes 1-{E(i)}]w_i\\
  &=&\theta_i(a).
  \end{eqnarray*}
  
   It suffices to show (i)-(iii), since also imply that $\gA\subseteq\gD$, finishing the
proof.
  
   For (i) we need to show that 
   \begin{equation}\label{star}
   \theta(a_1\cdots a_n)\widetilde{\zeta}=\theta(a_1\cdots a_{n-1})\theta(a_n)\widetilde{\zeta}
   \end{equation}
  for all $\widetilde{\zeta}\in E$ and all $a_k\in\gA_{i_k}$, $i_s\neq i_{s+1}$ and $i_n>i_{n-1}$.
  
   To simplify the notations, for $a\in\gA_i$, we will write $a^\circ=a-\phi_i(a)\in\gA_i^\circ$.
  
  If $\widetilde{\zeta}=\xi$, (\ref{star}) becomes
  \begin{equation*}
  v^\ast\sigma(a_1\cdots a_n)v\xi=v^\ast\sigma(a_1\cdots a_{n-1}vv^\ast \sigma(a_n)v\xi.
  \end{equation*}
But
\begin{eqnarray*}
\sigma(a_1\cdots a_n)v\xi
&=&
\sigma(a_1\cdots a_n) (1\otimes\xi)=\sigma(a_1)\cdots\sigma(a_{n})(1\otimes\xi)\\
&=&
\sigma(a_1)\cdots\sigma(a_{n-1})(a_n\otimes\xi)=\sigma(a_1)\cdots\sigma(a_{n-1})[a_n^\circ+\phi_n(a_n))\otimes\xi]\\
&=&
\sigma(a_1)\cdots\sigma(a_{n-1})(1\otimes\phi_n(a_n)\xi)
\end{eqnarray*}
   \end{proof}
since $a_n^\circ\otimes\xi\in E_n^\circ$ and $\phi_n(a_n)\otimes\xi= 1\otimes\phi_n(a_n)\xi$. On the other hand
\begin{eqnarray*}
\sigma(a_1\cdots a_{n-1})vv^\ast\sigma(a_n)v\xi
&=&
\sigma(a_1\cdots a_{n-1})vv^\ast(a_n\otimes\xi)\\
&=&
\sigma(a_1\cdots a_{n-1})(1\otimes\widehat{\theta(a_n)})\\
&=&
\sigma(a_1\cdots a_{n-1})(1\otimes\widehat{\theta(a_n^\circ)}+1\otimes\phi_n(a_n)\xi)\\
&=&
\sigma(a_1\cdots a_{n-1})(1\otimes\phi_n(a_n)\xi)
\end{eqnarray*}  
 since $\widehat{\theta(a_n^\circ)}\in  E_n^\circ$.
   
   Suppose now that $\widetilde{\zeta}=\zeta_1\otimes\cdots \zeta_m$, with $\zeta_j\in E_{l_j}^\circ$, $l_1>\cdots>l_m$; we will use the notation $\widetilde{\zeta}^\prime$ for $\zeta_2\otimes\cdots \zeta_m$.
   
   If $l_1>i_n$, then both sides of (\ref{star}) are zero.
  If $l_1<i_n$, then \[\sigma(a_n)v\widetilde{\zeta}=\sigma(a_n)(1\otimes\xi_n)\otimes v\widetilde{\zeta}\]
   and the argument reduces to the case $\widetilde{\zeta}=\xi$.
   
   If $l_1=i_n$, then we have
   \begin{eqnarray*}
 \sigma(a_1\cdots a_{n-1})vv^\ast\sigma(a_n)v\xi
&=&  
\sigma(a_1\cdots a_{n-1})vv^\ast\sigma(a_n)(1\otimes\zeta_1)\otimes(v\widetilde{\zeta}^\prime)\\
&=&
\sigma(a_1\cdots a_{n-1})v\LARGE(\widehat{\theta(a_n)\zeta_1}\otimes\widetilde{\zeta}^\prime\LARGE)\\
&=&
\sigma(a_1\cdots a_{n-1})(1\otimes\widehat{\theta(a_n)\zeta_1})\otimes (v\widetilde{\zeta}^\prime)\\
&=&
\sigma(a_1\cdots a_{n-1})\langle\widehat{\theta(a_n)\zeta_1}, \xi_{i_n}\rangle (v\widetilde{\zeta}^\prime).
   \end{eqnarray*}
  
   On the ohter hand,
   \begin{eqnarray*}
   \sigma(a_1\cdots a_n)v^\ast\widetilde{\zeta}
   &=&\sigma(a_1\cdots a_{n-1})\sigma(a_n)\LARGE((1\otimes\zeta_1)\otimes(v\widetilde{\zeta}^\prime)\LARGE)\\
   &=&\sigma(a_1\cdots a_{n-1})\LARGE((a_n\otimes\zeta_1)\otimes(v\widetilde{\zeta}^\prime)\LARGE).
   \end{eqnarray*}
   
  $a_n\otimes\zeta_1$ decomposes as  $\langle a_n\otimes\zeta_1, 1\otimes\xi_{i_n}\rangle\xi_{i_n}+\eta$, with $\eta\in F_{i_n}^\circ$, therefore the equality above becomes
  \[ \sigma(a_1\cdots a_n)v^\ast\widetilde{\zeta}=\sigma(a_1\cdots a_{n-1})\langle a_n\otimes\zeta_1, 1\otimes\xi_{i_n}\rangle(v\widetilde{\zeta}^\prime).\]
  
  Since we defined $F_i$ as $\gA_i\otimes_{\pi_i\circ\theta_i}E_i$, we have that 
  $\langle a_n\otimes\zeta_1, 1\otimes\xi_{i_n}\rangle =\langle\widehat{\theta(a_n)}\zeta_1, \xi_{i_n}\rangle$, and the proof is property (i) is complete.
  
   For (ii), it suffices to prove the property for the biggest $k\in\{1, 2, \dots, n\}$ such that $i_{k-1}<i_k>i_{k+1}$; also, since (i) was proved, we can suppose that $i_{k+1}>\dots>i_n$. In this framework, we need to show that
   \[
   v^\ast\sigma(a_1\cdots a_n)v=v^\ast\sigma(a_1\cdots a_{k-1}\phi_{i_k}(a_k)a_{k+1}\cdots a_n)v,
   \]
 that is 
 \[v^\ast\sigma(a_1\cdots a_{k-1}a_k)\sigma(a_{k+1}\cdots a_n)v
 =v^\ast\sigma(a_1\cdots a_{k-1})\phi_{i_k}(a_k)\sigma(a_{k+1}\cdots a_n)v.
 \] 
  Since $i_k>i-{k+1}>\dots >i_n$, it follows that 
  $\sigma(a_{k+1}\cdots a_n)v\widetilde{\zeta}\in F(i_k)$, for all $\widetilde{\zeta}\in E$, hence the assertion is equivalent to the first three cases from the proof of property (i).
  
  For part (iii), we need to show that
  \[ v^\ast\sigma(a_1\cdots a_n)v=v^\ast\sigma(a_1)vv^\ast\sigma(a_2\cdots a_n)v
  \]
  whenever $i_1>i_2$. Since (i) and (ii) are proved, we can suppose that $i_1>i_2>\dots>i_n$.  In this framework we have that $\sigma(a_2\cdots a_n)v\widetilde{\zeta}\in F(i_1)$ for all $\widetilde{\zeta}\in E$, therefore it suffices to show that  
  \begin{equation}\label{star2}
  v^\ast\sigma(a_1)\eta=v^\ast\sigma(a_1)vv^\ast\eta\ \text{for all}\ \eta\in F(i_1).
  \end{equation}
  But $v^\ast\sigma(a_1)\eta=v^\ast(a_1\otimes\xi_{i_1})\otimes\eta=\theta(a_1)\xi_{i_1}\otimes v^\ast\eta$. Also, since $v^\ast\eta\in E(i_1)$, we have that $ v v^\ast\eta\in F(i_1)$, hence 
  \begin{eqnarray*}
  v^\ast \sigma(a_1)vv^\ast\eta&=& v^\ast \sigma(a_1)[(1\otimes\xi_{i_1})\otimes vv^\ast\eta]\\
  &=& v^\ast [(a_1\otimes\xi_{i_1})\otimes vv^\ast\eta]=\theta(a_1)\xi_{i_1}\otimes v^\ast v v^\ast \eta \\
  &=& \theta(a_1)\xi_{i_1}\otimes v^\ast \eta, \ \text{since}\ v^\ast v=\text{Id}.
  \end{eqnarray*}

\end{document}